\documentclass[12pt]{amsart}

\usepackage{amsmath}
\usepackage{amsfonts}
\usepackage{amssymb,enumerate}
\usepackage{amsthm}
\usepackage{fancyhdr}
\usepackage{tikz}
\usepackage{verbatim}
\usetikzlibrary{arrows,matrix}
\usepackage{hyperref}

\usepackage{caption}
\usepackage[noend]{algpseudocode}
\usepackage{epstopdf}

\newtheorem{lemma}{Lemma}

\newtheorem{proposition}[lemma]{Proposition}
\newtheorem{theorem}[lemma]{Theorem}
\newtheorem{question}[lemma]{Question}
\newtheorem{example}[lemma]{Example}
\newtheorem{definition}[lemma]{Definition}

\title[]{On the type of an almost Gorenstein monomial curve}

\author{Alessio Moscariello}



\address[Alessio Moscariello]{Dipartimento di Matematica e Informatica, \ Universit\`a di Catania, \  Viale Andrea Doria 6, 
95125 Catania,Italy}

\address[Alessio Moscariello]{Scuola Superiore di Catania, \ Universit\`a di Catania, \  Via Valdisavoia 9, 
95125 Catania, Italy.}

\email{alessio.moscariello@studium.unict.it}

\begin{document}
\maketitle
\begin{abstract}
We prove that the Cohen-Macaulay type of an almost Gorenstein monomial curve $\mathcal C \subseteq \mathbb{A}^4$ is at most $3$, and make some considerations on the general case.

\end{abstract}
\section*{Introduction}
Almost Gorenstein rings have been introduced by Barucci and Fr\"{o}berg (cf. \cite{BF}) as a larger class of Cohen-Macaulay rings that are next to Gorenstein. In the same work, the authors proved some results for this class of rings, that found applications in \cite{KM}. The original definition was given for one-dimensional analytically unramified local rings; however recently Goto et al. (cf. \cite{G}) adapted this definition in order to deal with local Cohen-Macaulay rings of arbitrary dimension.

This work is focused on investigating possible bounds for the Cohen Macaulay type of local rings associated to almost Gorenstein monomial curves, in function of the embedding dimension. It is well-known that for one-dimensional analytically unramified local rings with embedding dimension $3$, not necessarily almost Gorenstein, the Cohen Macaulay type does not exceed $2$ (cf. \cite{FGH}, Theorem 11). However, in the same paper it has been showed that, if the embedding dimension is greater than $3$, there is no upper bound for the type. Thus the smallest interesting case is that of the coordinate ring of an almost Gorenstein monomial curve in $\mathbb{A}^4$. In this setting, further motivation for this work arises from a question by Numata (cf. \cite{N}), which we prove with the following:
\begin{theorem}\label{main}
The Cohen-Macaulay type of an almost Gorenstein monomial curve $ \mathcal C \subseteq \mathbb{A}^4$ is at most 3.
\end{theorem}
Many examples are present in literature (cf. \cite{N}) of almost Gorenstein monomial curves $ \mathcal C \subseteq \mathbb{A}^4$ with type $3$; therefore, this bound is sharp. 

The first section of this paper is devoted to proving Theorem \ref{main}, while in Section 2 we provide computational evidence and theoretical considerations for higher embedding dimensions. To simplify the exposition we will use the language of numerical semigroups (cf. \cite{RG}). Given the correspondence between numerical semigroups and monomial curves (cf. \cite{BDF}), in order to prove Theorem $1$ it suffices to prove that the type of a $4$-generated almost symmetric numerical semigroup is at most $3$.

\section{Main result}
Denote by $\mathbb{Z}$ and $\mathbb N$ the set of integers and nonnegative integers respectively. Given $e \ge 2$ and $n_1,n_2,\ldots,n_e\in\mathbb{N}$ such that $\gcd(n_1,n_2,\ldots,n_e)=1$, the \emph{numerical semigroup} generated by $\{n_1,n_2,\ldots,n_e\}$ is the set $$S=\langle n_1,n_2,\ldots,n_e\rangle=\{a_1n_1+a_2n_2+ \ldots + a_en_e \mid a_i \in\mathbb{N} 
\},$$ which is a submonoid of $(\mathbb N,+)$ such that $\mathbb{N} \setminus S$ is finite. With the notation $S=\langle n_1,n_2,\ldots,n_e\rangle$ we will assume that $\{n_1,n_2,\ldots,n_e\}$ is a minimal generating system for $S$; we will say that $e$ is the \emph{embedding dimension} of $S$, denoted by $e(S)$. We also denote by $F(S)$ the \emph{Frobenius number} of $S$, that is, $F(S)=\max \mathbb{Z} \setminus S$, and by $PF(S)$ the set of \emph{pseudo-Frobenius} numbers of $S$, $$PF(S) = \{x \not \in S \mid x+s \in S \text{ for every } s \in S \setminus \{0\}\}=\{x \not \in S \mid x+n_i \in S \text{ for every } i=1,\ldots,e\},$$ whose cardinality is called the \emph{type} of $S$, denoted by $t(S)$.

Let $\le_S$ be the relation defined by $x \le_S y$ if $y-x \in S$. It is easy to see that $(\mathbb{Z},\le_S)$ is a partially ordered set, and that the pseudo-Frobenius numbers of $S$ are the maximal elements of the poset $(\mathbb{Z} \setminus S, \le_S)$.

We say that a numerical semigroup $S$ is \emph{almost symmetric} (cf. \cite{BF}) if for every $x \in \mathbb{Z} \setminus S$ such that $F(S)-x \not \in S$ we have $\{x,F(S)-x\} \subseteq PF(S)$. 

We introduce some new objects associated to pseudo-Frobenius numbers, whose properties shed some light on the behaviour of almost symmetric numerical semigroups with $e(S)=4$, while also giving some insight on the generic case.
 
Notice that if $f \in PF(S)$, then $f + n_i \in S$ for every $i=1,\ldots,e$, hence there exist $a_{i1},\ldots,a_{ie} \in \mathbb{N}$ such that $$f+n_i=\sum_{j=1}^{e}a_{ij}n_j.$$
However $a_{ii} > 0$ would imply $f \in S$; thus $a_{ii}=0$. Thus, for every $i$, there exist $a_{i1},\ldots,a_{ie} \in \mathbb{N}$  such that $f=\sum_{j = 1}^{e} a_{ij}n_j$ and $a_{ii}=-1$.

\begin{definition}
Let $S = \langle n_1,\ldots,n_e \rangle$ be a numerical semigroup and $f \in PF(S)$. We say that $A=(a_{ij}) \in M_e(\mathbb{Z})$
is an \textbf{RF-matrix} (short for row-factorizazion matrix) for $f$ if $a_{ii}=-1$ for every $i=1,2,\ldots,e$, $a_{ij} \in \mathbb{N}$ if $i \neq j$ and for every $i=1,\ldots,e$ $$\sum_{j=1}^{e} a_{ij}n_j = f.$$
\end{definition}
\noindent

Notice that if $S$ is almost symmetric and $f \in PF(S) \setminus \{F(S)\}$, there exists an RF-matrix for both $f$ and $F(S)-f$. However, in general this matrix is not unique.
\begin{example}\label{example}
Consider the numerical semigroup $S= \langle 6,7,9,10 \rangle$. The pseudo-Frobenius numbers of $S$ are $PF(S)= \{3,8,11\}$, and thus $S$ is almost symmetric with type $3$. Now, take $8 \in PF(S)$. We have $8=3 \cdot 6 - 10 = 2 \cdot 9 - 10$, and thus the matrices
$$A_1=\begin{pmatrix}
-1 & 2 & 0 & 0 \\
1 & -1 & 1 & 0 \\
0 & 1 & -1 & 1 \\
3 & 0 & 0 & -1 \\ \end{pmatrix}, \ \ \ A_2=\begin{pmatrix}
-1 & 2 & 0 & 0 \\
1 & -1 & 1 & 0 \\
0 & 1 & -1 & 1 \\
0 & 0 & 2 & -1 \\ \end{pmatrix}$$
are both RF-matrices for $8$. 
	
\end{example}

\noindent
Notice that, in the example, $M_2$ and $M_1$ differ only by the fourth row, and that the two rows $(3,0,0,-1)$ and $(0,0,2,-1)$ are obtained from factorizations of $8+10$ as $3 \cdot 6$ and $2 \cdot 9$.

In the general case $S=\langle n_1, \ldots, n_e \rangle$, the $i$-th row of an RF-matrix for $f$ is associated to a factorization of $f+n_i$ in $S$; thus, denoting with $\mathsf Z (s)$ the set of factorizations of an element $s \in S$ as a linear combination of the minimal generators of $S$, we can choose the $i$-th row of an RF-matrix for $f$ in $|\mathsf Z (f+n_i)|$ ways. Thus the number of RF-matrices for $f \in PF(S)$ is equal to $$\prod_{i=1}^{e} |\mathsf Z (f+n_i)|.$$

If two elements of $PF(S)$ are \textit{symmetric} (that is, their sum is equal to $F(S)$), their RF-matrices gain a nice property:

\begin{proposition}\label{prop}

Let $S=\langle n_1,\ldots,n_e \rangle$ be a numerical semigroup, and let $f \in PF(S) \setminus \{F(S)\}$ be such that $F(S)-f \in PF(S)$. Let $A=(a_{ij})$ be an RF-matrix for $f$ and $B=(b_{ij})$ be an RF-matrix for $F(S)-f$. Then for every $i \neq j$ we have $a_{ij}b_{ji} = 0$.
\end{proposition}

\begin{proof}

Let $i,j \in \{1,\ldots,e\}$, $i \neq j$. Consider the $i$-th row of $A$ and the $j$-th row of $B$. By adding them we get
$$F(S)=f+(F(S)-f)=\sum_{p=1}^{e} a_{ip}n_p + \sum_{q=1}^{e}b_{jq}n_q=\sum_{h=1}^{e} (a_{ih}+b_{jh})n_h.$$
 
If $h \neq i,j$, it is clear that $a_{ih},b_{jh} \in \mathbb{N}$ and thus $a_{ih}+b_{jh} \in \mathbb{N}$, while the coefficients of $n_i$ and $n_j$ are respectively $b_{ji}-1$ and $a_{ij}-1$. Since $F(S) \not \in S$, these coefficients cannot be both non-negative: since $b_{ji},a_{ij} \in \mathbb{N}$ we necessarily have $b_{ji}=0$ or $a_{ij}=0$, that is $b_{ji}a_{ij}=0$. 
\end{proof}
Given $i,j \in \{1,\ldots, e(S)\}$, $i \neq j$, denote $$\lambda_{ij}=max\{K \in \mathbb{N} \ | \ Kn_j-n_i \not \in S \}, \ \ M_{ij}=\lambda_{ij}n_j-n_i \not \in S.$$ 
Denote by $\Lambda$ the multiset $\Lambda=\{M_{ij} | i \neq j\}$.
It is trivial to see that $K n_j - n_i \not \in S$ if $K \le \lambda_{ij}$ and $Kn_j-n_i \in S$ if $K > \lambda_{ij}$.

Let $f \in PF(S) \setminus \{F(S)\}$. Define the multiset $$\Gamma_f = \{ M_{ij} \in \Lambda \ | \ M_{ij} = f  \}$$ and let $\Gamma$ be the union of $\Gamma_f$ for every $f \in PF(S) \setminus \{F(S)\}$.

\begin{proposition}\label{unique}
Let $S$ be a numerical semigroup and $f \in PF(S)$ be such that $f=a n_j - n_i$, $j \neq i$.
Then $a= \lambda_{ij}$.
Furthermore, if $f,f' \in PF(S) \setminus \{F(S)\}$, $f \neq f'$, then $\Gamma_f \cap \Gamma_{f'} = \emptyset$.
\end{proposition}

\begin{proof}

Since $f \in PF(S)$ it follows that $f=an_j-n_i \not \in S$, therefore $a \le \lambda_{ij}$ and $f \le_S f + (\lambda_{ij}-a)n_j = M_{ij}$. However, from $M_{ij} \not \in S$ and $f \in PF(S)$ we deduce $f=M_{ij}$, that is $a=\lambda_{ij}$. The second part is obvious.
\end{proof}

In the rest of this Section we will consider the case of almost-symmetric numerical semigroups $S$ with $e(S)=4$.
The main idea is a counting argument on the number and placement of zeroes in RF-matrices; we will use the next lemmas to relate the sets $PF(S) \setminus \{F(S)\}$ and $\Lambda$ (we know that $|\Lambda|=12$ and $\Gamma \subseteq \Lambda$), thus bounding $t(S)$. In fact, if $A$ is an RF-matrix for $f \in PF(S) \setminus \{F(S)\}$, then each row of $A$ with exactly two zeroes gives an element of $\Lambda$ equal to $f$. However, $\Gamma_f$ may contain more elements than such rows of $A$: for example, in the numerical semigroup $S=\langle 6,7,9,10 \rangle$ considered in Example \ref{example}, the two RF-matrices $A_1$ and $A_2$ for $8$ both contain two rows with exactly two zeros, but clearly $Gamma_8$ contain more than two elements.

 In the proof of the next Lemma, we will use the well-known fact (cf. \cite{RG}, Corollary 10.22) that the type of any $3$-generated numerical semigroup is at most $2$.
\begin{lemma}\label{zeromatrix}
Let $S=\langle n_1,n_2,n_3,n_4 \rangle$ be an almost symmetric numerical semigroup. Let $f \in PF(S) \setminus \{F(S)\}$, and let $M$ be an RF-matrix for $f$ such that in a column of $M$ there is no positive element. Then $f=\frac{F(S)}{2}$.
\end{lemma}
\begin{proof}
Assume without loss of generality that $$M=\begin{pmatrix}
-1 & m_{12} & m_{13} & m_{14} \\
0 & -1 & m_{23} & m_{24} \\
0 & m_{32} & -1 & m_{34} \\
0 & m_{42} & m_{43} & -1 \\ \end{pmatrix}. $$
Let $d=\gcd(n_2,n_3,n_4)$. By making the first two rows of $M$ equal we obtain $$f=-n_1+m_{12}n_2+m_{13}n_3+m_{14}n_4=-n_2+m_{23}n_3+m_{24}n_4,$$ hence  $d | f$, and $d | n_1$; that implies $d | \gcd(n_1,n_2,n_3,n_4)=1$, and $d=1$. Thus $T=\langle n_2,n_3,n_4 \rangle$ is also a numerical semigroup, and by checking the RF-matrix $M$ we can deduce that $f \not \in T$ and $f+n_2,f+n_3,f+n_4 \in T$, that is $f \in PF(T)$. $T \subseteq S$ implies $F(T) \ge F(S)$, thus $f \neq F(T)$. Then $t(T) = 2$ and $\{f,F(T)\} = PF(T)$.

Consider now $F(T)-f$. Since $f \in PF(T) \setminus \{F(T)\}$, we have $F(T)-f \not \in T$; by definition of pseudo-Frobenius number we must have either $F(T)-f \le_T F(T)$ or $F(T)-f \le_T f$. Since $F(T)-f \le_T F(T)$ would imply the contradiction $f \in T$, we can deduce $F(T)-f \le_T f$.

On the other hand, considering $F(S) \not \in T$, since $f \in PF(S) \setminus \{F(S)\}$, $F(S) > f$ implies  $F(S) \le_T F(T)$, and thus $F(T)-F(S) \in T$.
But since $F(S)-f \not \in S$, we have $F(S)-f \not \in T$, and  $$F(S)-f \le_T F(S)-f +(F(T)-F(S))=F(T)-f \le_T f,$$ that is $F(S)-f \le_T f$, and clearly $F(S)-f \le_S f$. Finally, since $F(S)-f \in PF(S)$ we must have $F(S)-f=f$, and thus our claim.
\end{proof}

\begin{lemma} \label{zeronumber}
Let $S=\langle n_1,n_2,n_3,n_4 \rangle$ be almost symmetric, $f \in PF(S) \setminus \{F(S)\}$, and let $A$ and $B$ be RF-matrices respectively for $f$ and $F(S)-f$.
Then there are at most $|\Gamma_f| + |\Gamma_{F(S)-f}| + 8$ zeroes in $A \cup B$.
\end{lemma}

\begin{proof}
Since $f \in \mathbb{N}$ there cannot be a row of $A$ or $B$ with three zeroes: thus, denoting with $m_1$ the number of rows of both $A$ e $B$ having exactly one zero and $m_2$ the number of rows of both $A \cup B$ having exactly two zeroes, $m_1+m_2 \le 8$. Also, there are exactly $m_1+2m_2$ zeroes in $A \cup B$. 
Notice that if a row of either $A$ or $B$ has exactly two zeroes, then by Proposition \ref{unique} this row corresponds to an element of $\Lambda$ equal to $f$ or $F(S)-f$. Then $m_2 \le |\Gamma_f| + |\Gamma_{F(S)-f}|$, $m_1 \le 8 - m_2$. The maximum possible value of $m_1+2m_2$ under these restrictions is $|\Gamma_f| + |\Gamma_{F(S)-f}| + 8$, that is our conclusion.
\end{proof}

\begin{lemma} \label{onetwo}
Let $S=\langle n_1,n_2,n_3,n_4 \rangle$ be almost symmetric, and $f \in PF(S) \setminus \{F(S)\}$.
Then $|\Gamma_f| + |\Gamma_{F(S)-f}| \ge 4$. Moreover $|\Gamma| \ge 2|PF(S)\setminus \{F(S)\}| = 2(t(S)-1)$.
\end{lemma}

\begin{proof}
Let $A=(a_{ij})$ be an RF-matrix for $f$ and $B=(b_{ij})$ an RF-matrix for $F(S)-f \in PF(S) \setminus \{F(S)\}$. In these matrices there are $12$ pairs of elements of the form $b_{ij},a_{ji}$: then by Proposition \ref{prop} there are at least $12$ zeroes among the elements of $A$ and $B$. Thus by Lemma \ref{zeronumber} it follows $12 \le |\Gamma_f| + |\Gamma_{F(S)-f}| + 8$, that is $|\Gamma_f| + |\Gamma_{F(S)-f}| \ge 4$.

Since the various $\Gamma_f$ are disjoint, noticing that for $f=\frac{F(S)}{2}$  we have $|\Gamma_f| \ge 2$, by adding the various $|\Gamma_f|$ and pairing $|\Gamma_f|$ with $|\Gamma_{F(S)-f}|$ we obtain $|\Gamma| \ge 2|PF(S)\setminus \{F(S)\}| = 2(t(S)-1)$.
\end{proof}

In the previous Lemma we related the zeroes of a pair of RF-matrices with elements of $\Gamma$ and with the type $t(S)$. This result alone provides a first bound for $t(S)$: in fact, since $\Gamma \subseteq \Lambda$ it is clear that $|\Lambda| = 12 \ge 2(t(S)-1)$, thus we can deduce $t(S) \le 7$.

Proposition \ref{prop} guarantees the existence of at least $12$ zeroes in $A \cup B$. Next, we will see that this bound can be improved depending on the behaviour of the elements of $\Gamma$.

\begin{lemma}\label{zeroplace}

Let $S= \langle n_1,n_2,n_3,n_4 \rangle$ and consider $f,f' \in PF(S) \setminus \{F(S)\}$, with $$f=M_{ji}=\lambda_{ji} n_i-n_j, \ \ f'=M_{ki}=\lambda_{ki}n_i-n_k,\ \  \lambda_{ji} \ge \lambda_{ki}$$ for three distinct indexes $i,j,k \in \{1,2,3,4\}$. \\
Let $A=(a_{pq})$ be an RF-matrix for $F(S)-f$. Then $a_{kj}=0$.
\end{lemma}

\begin{proof}

Assume that $a_{kj} \neq 0$. Taking the $k$-th row of $A$ we get  $$F(S)=f+(F(S)-f)=\lambda_{ji}n_i-n_j+a_{ki}n_i+a_{kj}n_j-n_k+a_{kh}n_h=$$ $$=(\lambda_{ji}+a_{ki})n_i+(a_{kj}-1)n_j-n_k+a_{kh}n_h.$$
Since $\lambda_{ji}+a_{ki} \ge \lambda_{ji} \ge \lambda_{ki}$ and $a_{kj}-1 \ge 0$, it follows $f'=\lambda_{ki}n_i-n_k \le_S F(S)$, that is a contradiction.
\end{proof}

The meaning of Proposition \ref{zeroplace} is that for each $M_{ji},M_{ki} \in \Gamma$ we can find a pair $a_{jk},b_{kj}$ such that $a_{jk}=b_{kj}=0$, thus adding one more zero to the lower bound on the nuumber of zeroes provided by Proposition \ref{prop}. The next step concerns possible configurations of the elements of $\Gamma$.

\begin{proposition}\label{triplet}
Let $S=\langle n_1,n_2,n_3,n_4 \rangle$ be almost symmetric. Then there exist no distinct $f,f',f'' \in PF(S) \setminus \{F(S)\}$ such that
$$f=M_{ji}=\lambda_{ji} n_i-n_j, \ \ f'=M_{ki}=\lambda_{ki}n_i-n_k, \ \ f''=M_{hi}=\lambda_{hi}n_i-n_h,\ \  \{i,j,k,h\}=\{1,2,3,4\}.$$
\end{proposition}

\begin{proof}
Assume, without loss of generality, that $$f=M_{21}=\lambda_{21} n_1-n_2, \ \ f'=M_{31}=\lambda_{31}n_1-n_3, \ \ f''=M_{41}=\lambda_{41}n_1-n_4$$ and that $\lambda_{21} \ge \lambda_{31} \ge \lambda_{41}$.
Take $g=F(S)-f \in PF(S)$, and consider the RF-matrices $$F=\begin{pmatrix}
-1 & f_{12} & f_{13} & f_{14} \\
\lambda_{21} & -1 & 0 & 0 \\
f_{31} & f_{32} & -1 & f_{34} \\
f_{41} & f_{42} & f_{43} & -1 \\ \end{pmatrix}, \ \ \ G=\begin{pmatrix}
-1 & g_{12} & g_{13} & g_{14} \\
g_{21} & -1 & g_{23} & g_{24} \\
g_{31} & g_{32} & -1 & g_{34} \\
g_{41} & g_{42} & g_{43} & -1 \\ \end{pmatrix}$$  respectively for $f$ and $g$. 
Since $\lambda_{21} > 0$ it follows that $g_{12}=0$. Applying Proposition \ref{zeroplace} to the pairs $\{f,f'\}$ and $\{f,f''\}$ we obtain $g_{32}=g_{42}=0$.
Thus $G$ satisfies the hypotheses of Lemma \ref{zeromatrix}, hence $g=\frac{F(S)}{2}$, and $f=\frac{F(S)}{2}$.

$f=g$ implies that $G$ is an RF-matrix for both $g$ and $f$, thus applying Proposition \ref{prop} to $G$, considered as RF-matrix for $f$ and $g$, we obtain that $g_{ij} \neq 0$, which implies $g_{ji} = 0$.

Finally, since $f=g=\lambda_{21}n_1-n_2$ we can assume without loss of generality (up to switching the second rows of $F$ and $G$) that
$$G=\begin{pmatrix}
-1 & 0 & g_{13} & g_{14} \\
\lambda_{21} & -1 & 0 & 0 \\
g_{31} & 0 & -1 & g_{34} \\
g_{41} & 0 & g_{43} & -1 \\ \end{pmatrix}.$$ 
However, the implication on $G$ assures that at least one between $g_{34},g_{43}$ is zero.
Assuming $g_{43}=0$, we obtain $$f=\lambda_{21}n_1-n_2=g_{41}n_1-n_4,$$ and since $f=g_{41}n_1-n_4$ by Proposition \ref{unique} $g_{41}=\lambda_{41}$.
Therefore $\lambda_{21} \ge \lambda_{41} = g_{41}$, and $n_2=(\lambda_{21}-g_{41})n_1+n_4$, that is impossible.
Assuming $g_{34}=0$ a similar reasoning leads to another contradiction.
\end{proof}

Taking into account these results, we can improve our bound for $t(S)$.

\begin{proposition}\label{type4}
Let $S= \langle n_1,n_2,n_3,n_4 \rangle$ be an almost symmetric numerical semigroup. Then $t(S) \le 4$.
\end{proposition}

\begin{proof}

By Lemma \ref{onetwo}, $|\Gamma| \ge 2(t(S)-1)$. However, $|\Gamma| \ge 9$ implies that there are three elements of the form $M_{ji},M_{ki},M_{hi}$ in $PF(S) \setminus \{F(S)\}$ for $\{i,j,k,h\}=\{1,2,3,4\},$
contradicting Lemma \ref{triplet}. Thus $|\Gamma| \le 8$, and $2(t(S)-1) \le 8$, that is $t(S) \le 5$.

Assume now $t(S)=5$, which implies $|\Gamma| = 8$. Then Lemma \ref{triplet} forces that for every $i=1,2,3,4$  there exist exactly two indexes $j,k$ such that $M_{ji},M_{ki} \in \Gamma$. 
Thus we can assume, without loss of generality, that there exist $f,f' \in PF(S) \setminus \{F(S)\}$ such that $$f=\lambda_{21}n_1-n_2, \ \ \ f'=\lambda_{31}n_1-n_3, \ \ \lambda_{21} \ge \lambda_{31}.$$
Let $A=(a_{ij})$ be an RF-matrix for $f$ and $B=(b_{ij})$ be an RF-matrix for $F(S)-f$.
Lemma \ref{onetwo} and $|\Gamma|=8$ imply that $|\Gamma_f| + |\Gamma_{F(S)-f}| = 4$ for every $f \in PF(S) \setminus \{F(S)\}$, therefore by Lemma \ref{zeronumber} and Proposition \ref{prop} we deduce that $A \cup B$ contain exactly $12$ zeroes, and thus, by Proposition \ref{prop}, for every pair of indexes $i,j$ with $i \neq j$, exactly one between $a_{ij},b_{ji}$ is zero. However by Proposition \ref{zeroplace} we have $b_{32}=0$, thus taking the pair of elements $0=a_{23}=b_{32}$ we reach a contradiction.
\end{proof}
The final step of our proof excludes the case $t(S)=4$. This case is somewhat more complicated, and our argument is slightly different:

\begin{lemma}\label{4prop}

Let $S= \langle n_1,n_2,n_3,n_4 \rangle$ be almost symmetric, with $t(S)=4$, and let $C$ be an RF-matrix for $\frac{F(S)}{2} \in PF(S) \setminus \{F(S)\}.$
Then $|\Gamma|=8$, $|\Gamma_{\frac{F(S)}{2}}|=4$, and each row of $C$ contains exactly two zeroes.
\end{lemma}

\begin{proof}
By Lemma \ref{onetwo} and Lemma \ref{triplet} it follows $6 \le |\Gamma| \le 8$. Let $PF(S) = \{f,\frac{F(S)}{2},F(S)-f,F(S)\}$, and let $A$, $B$, be RF-matrices for $f$ e $F(S)-f$. 
Consider the two pairs of matrices $A,B$ and $C,C$, and let $\Gamma_A,\Gamma_B,\Gamma_C$ be the (disjoint) sets containing the elements of $\Gamma$ that appear in the matrices $A,B,C$ respectively. There is a bijection between these three sets and the rows of $A,B,C$ having exactly two zeroes. Moreover, the number of rows of $A$ with exactly one zero is at most $4-|\Gamma_A|$, and then $A$ has at most $4 -|\Gamma_A| + 2|\Gamma_A|=4+|\Gamma_A|$ zeroes (and similarly for $B$ and $C$). Summing these values we obtain that there are at most $16+|\Gamma_A|+|\Gamma_B|+2|\Gamma_C|$ zeroes in the two pairs of matrices $A,B$ and $C,C$.

However, counting these zeroes starting from Proposition \ref{prop}, there are at least $12$ zeroes for each pair of matrices. By Lemma \ref{triplet}, for each $i$ there can be at most two indexes $j,k$ and two elements $f_i,f_i' \in PF(S) \setminus \{F(S)\}$ such that $f_i=M_{ji}$,$f_i'=M_{ki}$. Furthermore, for each such pair, Proposition \ref{zeroplace} states that there is a pair of elements either of the form $a_{jk},b_{kj}$ or $c_{jk},c_{kj}$ that are both zeroes. Since the number of such pairs is at least $|\Gamma_A|+|\Gamma_B|+|\Gamma_C|-4$, we obtain that there are at least $12+12+|\Gamma_A|+|\Gamma_B|+|\Gamma_C|-4$ zeroes in the two pairs $A,B$ and $C,C$.
Combining both bounds we obtain the inequality
 $$20+|\Gamma_A|+|\Gamma_B|+|\Gamma_C| \le 16+|\Gamma_A|+|\Gamma_B|+2|\Gamma_C|,$$ hence $|\Gamma_C| \ge 4.$
 This implies that each row of $C$ represents an element of $\Gamma_C$ and thus contains exactly two zeroes, and the restrictions $\Gamma_C \subseteq \Gamma_{\frac{F(S)}{2}}$, $|\Gamma| \le 8$ and $|\Gamma_f|+|\Gamma_{F(S)-f}| \ge 4$ force $|\Gamma|=8$, $|\Gamma_{\frac{F(S)}{2}}|=4$.
\end{proof}

With this Lemma we are ready to prove that $t(S) \le 3$. Here, our proof revolves around showing that if $t(S)=4$ then, by using the previous Lemma, there is no RF-matrix for $\frac{F(S)}{2} \in PF(S) \setminus \{F(S)\}$, thus reaching the desired contradiction.

\begin{proof}[Proof of Theorem 1]

Assume that $t(S) = 4$, that is $|PF(S) \setminus \{F(S)\}|=3$. Since Lemma \ref{4prop} states that $|\Gamma|=8$, by Lemma \ref{triplet} it follows that for each index $i$ there exist exactly two elements $M_{j_i,i},M_{k_i,i}$ such that $M_{j_i,i},M_{k_i,i} \in PF(S) \setminus \{F(S)\}$, and $$M_{j_i,i}=\lambda_{j_i,i}n_i-n_{j_i}, \ \ M_{k_i,i}=\lambda_{k_i,i}n_i-n_{k_i}, \ \ \lambda_{j_i,i} \ge \lambda_{k_i,i}.$$
Take now $A$ and $B$ RF-matrices for $f$ and $F(S)-f$.
By Lemma \ref{zeronumber} and Proposition \ref{prop} there are exactly $12$ zeroes in the pair of matrices $A,B$. However, $M_{j_i,i} \neq \frac{F(S)}{2}$ would imply by Proposition \ref{zeroplace} that there are at least $13$ zeroes in the pair of matrices $A,B$. Therefore $M_{j_i,i}=\frac{F(S)}{2}$ for every $i=1,2,3,4$, and $\lambda_{j_i,i} > \lambda_{k_i,i}$.

Let $C$ be an RF-matrix for $\frac{F(S)}{2}$; recall that by Lemma \ref{4prop} each row of $C$ contains exactly two zeroes, and all elements of $\Gamma_{\frac{F(S)}{2}}$ appear in $C$. We will now show that there cannot be such a matrix.

Assume, rearranging our indexes, that $j_1=2$ and $k_1=3$, that is $\lambda_{21} \ge \lambda_{31}$ and $M_{21}=\lambda_{21}n_1-n_2=\frac{F(S)}{2}$. Then we have $$C=\begin{pmatrix}
-1 & c_{12} & c_{13} & c_{14} \\
\lambda_{21} & -1 & 0 & 0 \\
c_{31} & c_{32} & -1 & c_{34} \\
c_{41} & c_{42} & c_{43} & -1 \\ \end{pmatrix}.$$
If $c_{p1} \neq 0$ for some $p \neq 2$, since each row of $C$ has two zeroes then (Proposition \ref{unique}) $c_{p1}=\lambda_{p1}$, hence $\lambda_{21}n_1-n_2=\lambda_{p1}n_1-n_p$, that leads to either $n_2 \in \langle n_1,n_p \rangle$ or $n_p \in \langle n_1,n_2 \rangle$. However, both conclusions are impossible; thus $c_{41}=c_{31}=0$. With a similar reasoning, we can prove that in each column of $C$ there can be at most one positive element, and since in each row of $C$ there are at least two zeroes (hence there are at most four positive elements in all $C$), it follows necessarily that in each column of $C$ there is exactly one positive element.

Moreover, by Proposition \ref{prop} $c_{12}=0$ and by Proposition \ref{zeroplace} $c_{32}=0$. Thus we must have $c_{42} > 0$; by Proposition \ref{unique} we deduce $c_{42}=\lambda_{42}$ and $c_{41}=c_{43}=0$. Moreover, since $c_{31}=c_{32}=0$, considering the third row of $C$ we must have $c_{34} > 0$, that is $c_{34} = \lambda_{34}$. Finally, $c_{34} > 0$ implies $c_{14}=0$, therefore considering the first row of $C$ we obtain $c_{13}=\lambda_{13}$. Thus in this setting $C$ is fully determined, and
$$C=\begin{pmatrix}
-1 & 0 & \lambda_{13} & 0 \\
\lambda_{21} & -1 & 0 & 0 \\
0 & 0 & -1 & \lambda_{34} \\
0 & \lambda_{42} & 0 & -1 \\ \end{pmatrix},$$
that is $M_{21}=M_{34}=M_{13}=M_{42}=\frac{F(S)}{2}$.

Now we are left to find the elements $M_{k_i,i}$. By Proposition \ref{zeroplace} on $\frac{F(S)}{2}$, there is a correspondence between the elements $M_{k_i,i}$ and the pairs of elements $c_{j_ik_i},c_{k_ij_i}$ in $C$ such that $c_{j_ik_i}=c_{k_ij_i}=0$. By checking $C$ we notice that the only such pairs are $c_{23}=c_{32}=0$ and $c_{14}=c_{41}=0$, thus we obtain that $j_1=2, j_2=4,j_3=1,j_4=3$ imply $k_1=3, k_2=1, k_3=4, k_4=2.$
Consider now the following four elements: $M_{21}=M_{34}=\frac{F(S)}{2}$, and $M_{31},M_{24} \in \Gamma_f \cup \Gamma_{F(S)-f}$. By definition of $j_i$ and $k_i$ we have $\lambda_{21} > \lambda_{31}$ and $\lambda_{34} > \lambda_{24}$.
 
Now, under the assumption that $M_{31}$ and $M_{24}$ do not belong to the same set $\Gamma_f$, we deduce $$M_{31}+M_{24}=F(S)=M_{21}+M_{34} \Longrightarrow \lambda_{31}n_1-n_2-n_3+\lambda_{24}n_4=\lambda_{21}n_1-n_2-n_3+\lambda_{34}n_4,$$ that is a contradiction since $\lambda_{21} > \lambda_{31}$ and $\lambda_{34} > \lambda_{24}$. Thus $M_{31}=M_{24}$, and since $M_{13}=M_{34}=\frac{F(\mathcal{S})}{2}$, we have $$M_{34}+M_{31}-M_{24}=M_{13}=\frac{F(\mathcal{S})}{2} \Longrightarrow$$
$$\lambda_{34}g_4-g_3+\lambda_{31}g_1-g_3-\lambda_{24}g_4+g_2=\lambda_{13}g_3-g_1=\frac{F(\mathcal{S})}{2}.$$
Then we have $$F(\mathcal{S})=\lambda_{34}g_4-g_3+\lambda_{31}g_1-g_3-\lambda_{24}g_4+g_2+\lambda_{13}g_3-g_1 \Longrightarrow$$ $$F(\mathcal{S})=(\lambda_{34}-\lambda_{24})g_4+(\lambda_{13}-2)g_3+g_2+(\lambda_{31}-1)g_1.$$
Since $j_4=3, k_4=2$ and $\lambda_{j_i,i} > \lambda_{k_i,i}$ for every $i=1,2,3,4$, it follows that $\lambda_{34} > \lambda_{24}$. Thus $F(\mathcal{S}) \not \in \mathcal{S}$ implies $\lambda_{13}=1$; on the other hand, we know that $j_3=1$ and $k_3=4$, therefore $1=\lambda_{13} > \lambda_{43} > 0$, which yields a contradiction.

Therefore there cannot be an RF-matrix $C$ for $\frac{F(S)}{2}$, and since $\frac{F(S)}{2} \in PF(S)$, this is a contradiction. Then $t(S) \le 3$.
\end{proof}

\section{Considerations on the general case}

Given the results of Section $1$, it is natural to investigate bounds for $t(S)$ in higher embedding dimension. In this work, bounds for $t(S)$ were proved by finding factorizations of pseudo-Frobenius numbers of $S$ as $\lambda n_i-n_j$, and this was done by counting zeroes in RF matrices for $\{f,F(S)-f\} \subseteq PF(S)$ via Proposition \ref{prop}, while sharper bounds were found using various lemmas.

If we consider an almost symmetric numerical semigroup $S=\langle n_1,\ldots, n_e \rangle$, $e \ge 4$, $f \in PF(S) \setminus \{F(S)\}$ and $A,B$ RF-matrices respectively for $f$ and $F(S)-f$, there are $e^2-e$ pairs of elements of the form $a_{ij},b_{ji}$, thus by Proposition \ref{prop} there are at least $e^2-e$ zeroes in $A \cup B$. If we denote by $r_i$ the number of rows of $A \cup B$ with at least $i$ zeroes, in order to relate $f$ or $F(S)-f$ with elements of the form $\lambda n_i-n_j$, we need some bounds for $r_{e-2}$. Since we have $$\sum_{i=1}^{e-2} r_i \ge e^2-e$$
and each $r_i$ is at most $2e$, setting  $d=\left \lfloor \frac{e}{2} \right \rfloor$ we have $$\sum_{i=1}^{e-2} r_i = \sum_{i=1}^{d-1} r_i + \sum_{i=d}^{e-2} r_i \le 2e(d-1) + \sum_{i=d}^{e-2} r_i \le e^2-2e + \sum_{i=d}^{e-2} r_i,$$ and thus by combining these inequalities we obtain $$\sum_{i=d}^{e-2} r_i \ge e.$$ However, while for $e=4$ we have $e-2=d$ and thus we can deduce $r_{e-2} \ge 4$, for $e \ge 5$ this deduction clearly fails. Furthermore, bounds for $r_p$ with $p < e-2$ are not as useful, since there could be more than one pseudo-Frobenius numbers of the form $\displaystyle (\sum a_i n_i)-n_j$ for the same set of indexes. Thus a generalization of our argument does not seem straightforward.

Using GAP (cf. \cite{D}), we computed $t(S)$ and $e(S)$ for almost symmetric numerical semigroups $S$ such that $|\mathbb{N} \setminus S| \le 32$ (around $10^6$ numerical semigroups). The results are summarized in Figure \ref{fig}.
\begin{figure}[h] 
  \centering 
  \includegraphics[scale=.6]{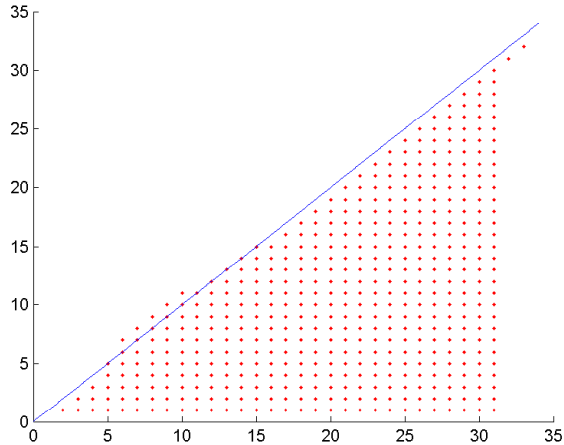} 
  \caption{Type and embedding dimension of $S$ such that $|\mathbb{N}\setminus S | \le 32$}\label{fig} 
\end{figure}

Our result proves that $t(S) < e(S)$ if the embedding dimension is four. However, this inequality fails if $e(S)=5$.
\begin{example}\label{t>e}
	Consider the numerical semigroup $S= \langle 14,15,17,19,20  \rangle$. We have $PF(S)= \{16,18,23,25,41\}$, and since $41=25+16=23+18$, $S$ is an almost-symmetric numerical semigroup such that $t(S)=e(S)=5$.
\end{example}
While in general the inequality $t(S) < e(S)$ fails, the data suggests that for each value of the embedding dimension the type is bounded. Recall that this is false for arbitrary numerical semigroup (cf. \cite{RG}, Example 2.24). Furthermore, in the examples considered, the type is always bounded by $e(S)+1$. In a private communication, Francesco Strazzanti pointed out that, if $\mathcal{S}$ is an almost symmetric numerical semigroup such that $t(\mathcal{S})-e(\mathcal{S}) \ge 0$ (like the one presented in Example \ref{t>e}), it is possible to construct, via \emph{numerical duplication} (cf. \cite{BDS}), an almost symmetric numerical semigroup $\mathcal{S}'$ such that $t(\mathcal{S}')-e(\mathcal{S}') > t(\mathcal{S})-e(\mathcal{S})$. Therefore, the inequality $t(\mathcal{S}) \le e(\mathcal{S})+b$, with $b$ a positive integer, does not hold for all almost symmetric numerical semigroups.

We conclude this work with the question:
\begin{question}
Let $S$ be an almost symmetric numerical semigroup. 

Is $t(S)$ bounded by a function of $e(S)$?

\end{question}

\section*{Acknowledgements}
I would like to thank Alessio Sammartano for his helpful comments and remarks on this work, and for his great help in the contextualization of this result, and Professor Marco D'Anna for some discussions on this subject. Moreover, I would like to thank Professor Keiichi Watanabe and Professor Naoyuki Matsuoka for their comments, and for noticing an error in the proof on the main result.

\end{document}